\documentclass{article}

\newtheorem{theorem}{Theorem}

\newtheorem{corollary}[theorem]{Corollary}

\newenvironment{proof}[1][Proof]{\noindent\textbf{#1.} }{\ \rule{0.5em}{0.5em}}
\input{tcilatex}

\begin{document}

\title{On Routh-Steiner Theorem and Generalizations\thanks{%
MSC: 51N10}}
\author{Elias Abboud}
\maketitle

\begin{abstract}
Following Coxeter we use barycentric coordinates in affine geometry to prove
theorems on ratios of areas.

In particular, we prove a version of Routh-Steiner theorem for
parallelograms.
\end{abstract}

\section{Introduction}

Coxeter in his book \cite[p. 211]{Coxeter}, considered the following theorem
of \textit{affine} geometry;

\begin{theorem}
\label{Routh th}If the sides $BC,CA,AB$ of a triangle $ABC$ are divided at $%
L,M,N$ $\ $in the respective ratios $\lambda :1,$ $\mu :1,$ $\nu :1,$ the
cevians $AL,BM,CN$ form a triangle whose area is 
\[
\frac{(\lambda \mu \nu -1)^{2}}{(\lambda \mu +\lambda +1)(\mu \nu +\mu
+1)(\nu \lambda +\nu +1)} 
\]

times that of $ABC.$
\end{theorem}

He emphasized that this result was discovered by Steiner, but simultaneously
cited two references the first was\ Steiner's work \cite[p. 163-168]{Steiner}
and the second was Routh's work: \cite[p. 82]{Routh}. Later in his book he
referred to the result as "Routh's theorem" \cite[p. 219]{Coxeter} admitting
to the contribution of both scientists in revealing the theorem.

Coxeter gave a general proof of this result using \textit{barycentric }%
coordinates attributed to M\"{o}bius\textit{.} These are \textit{homogeneous}
coordinates $(t_{1},t_{2},t_{3})$, where $t_{1},t_{2},t_{3}$ are masses at
the vertices of a triangle of reference $A_{1}A_{2}A_{3}$. In particular $%
(1,0,0)$ is $A_{1},$ $(0,1,0)$ is $A_{2}$, $(0,0,1)$ is $A_{3}$ and $%
(t_{1},t_{2},t_{3})$ corresponds to a point $P$ such that the areas of the
triangles $PA_{2}A_{3},PA_{3}A_{1},PA_{1}A_{2}$ are proportional to the
barycentric coordinates $t_{1},t_{2},t_{3}$ of $P$, respectively (see Fig. %
\ref{triangle-fig}). If $t_{1}+t_{2}+t_{3}=1$ then the normalized
barycentric coordinates $(t_{1},t_{2},t_{3})$ are called \textit{areal}
coordinates. In this case the areas of the triangles $%
PA_{2}A_{3},PA_{3}A_{1},PA_{1}A_{2}$ are $t_{1},t_{2},t_{3}$ times the area
of the whole triangle $A_{1}A_{2}A_{3}$, respectively.

Recently, A. B\'{e}nyi and B. \'{C}urgus \cite{Benyi-Curgus}, proved a
version of Theorem \ref{Routh th} and got a unification of the theorems of
Ceva and Menelaus. In fact they unified two expressions of Routh on area
ratios into one from which they derived both theorems of Ceva and Menelaus
as special cases.

Putting aside the search after "nice expressions", Theorem \ref{Routh th}
implies the following general property:

\begin{theorem}
\label{general division of triangle} If the sides of a triangle $%
A_{2}A_{3}A_{1}$ are divided at $A_{i,1},A_{i,2},...,A_{i,n-1},1\leq i\leq 3$
in the respective ratios $\lambda _{i,1}:\lambda _{i,2}:...:\lambda
_{i,n-1},1\leq i\leq 3,$ then the ratio of the area of any sub polygon to
the area of the whole triangle depends only on $\left\{ \lambda
_{i,j}\right\} ,$ where $1\leq i\leq 3$ and $1\leq j\leq n-1$ (see Fig. \ref%
{triangular-network} \ for the case $n=5).$
\end{theorem}

A sub polygon of the triangle is defined as a polygon whose vertices are
points of intersections of the cevians $\left\{ A_{i}A_{i,j}\right\} ,1\leq
i\leq 3,1\leq j\leq n-1$.

To prove the theorem we compute first the barycentric coordinates for each
vertex of the sub polygon; these are points of intersections of the
barycentric equations of the cevians $\left\{ A_{i}A_{i,j}\right\} ,1\leq
i\leq 3,1\leq j\leq n-1$. Next we divide the sub polygon into
non-overlapping triangles (in Fig. \ref{triangular-network} the polygon $%
MNOPQR$ is divided into $4$ triangles). Then we use the method which Coxeter
gave in his book \cite[p. 219]{Coxeter} for proving Theorem \ref{Routh th}
(which will be illustrated in the next section) to conclude that the ratio
of the area of each triangle to the area of the whole triangle depends only
on $\left\{ \lambda _{i,j}\right\} $ and the result follows.

\section{Patterns}

The search after nice expressions of area ratios succeeds in "symmetric"
divisions. In particular we have the following theorem:

\begin{theorem}
\label{1-lampda-1-th}Suppose the sides of a triangle $A_{2}A_{3}A_{1}$ are
divided at $A_{i,1},A_{i,2},1\leq i\leq 3$ in the respective ratios $%
1:\lambda :1$ (see Fig. \ref{1-lampda-1}). Let $I,J,K,L,M,N$ be the points
of intersection of the corresponding cevians as shown in the table:%
\[
\begin{array}{cc}
\text{point} & \text{cevians intersection} \\ 
I & A_{1}A_{1,1}\cap A_{3}A_{3,2} \\ 
J & A_{2}A_{2,1}\cap A_{3}A_{3,2} \\ 
K & A_{2}A_{2,1}\cap A_{1}A_{1,2} \\ 
L & A_{3}A_{3,1}\cap A_{1}A_{1,2} \\ 
M & A_{2}A_{2,2}\cap A_{3}A_{3,1} \\ 
N & A_{1}A_{1,1}\cap A_{2}A_{2,2}%
\end{array}%
\]

Then the ratio of the area of the hexagon $IJKLMN$ to the area of the
triangle $A_{2}A_{3}A_{1}$ is 
\[
\frac{2\lambda ^{2}}{(3+\lambda )(2\lambda +3)}. 
\]
\end{theorem}

\begin{proof}
At first note that such a hexagon exists in any triangle since it exists in
an equilateral triangle, by symmetry of the division ratios, and every other
triangle is affine equivalent to an equilateral triangle. Let $G$ be the
centre of gravity of $A_{2}A_{3}A_{1}.$ Since the barycentric coordinates of 
$A_{1},A_{2},A_{3}$ are $(1,0,0),(0,1,0),(0,0,1)$ respectively, the
barycentric coordinates of $G$ are $(\frac{1}{3},\frac{1}{3},\frac{1}{3}).$
By symmetry, the area of the hexagon $IJKLMN$ is $6$ times the area of the
triangle $GMN.$ Therefore, it is sufficient to compute the barycentric
coordinates of $M$ and $N.$

In order to find the barycentric equations of the corresponding cevians, we
exhibit the barycentric coordinates of some of the division points. Now, $%
A_{1,1}$ divides $A_{2}A_{3}$ in the ratio $1:\lambda +1,$ $A_{2,2}$ divides 
$A_{3}A_{1}$ in the ratio $\lambda +1:1$ and $A_{3,1}$divides $A_{1}A_{2}$
in the ratio $1:\lambda +1.$ Thus the barycentric coordinates are given in
the table: 
\[
\begin{array}{cc}
\text{point} & \text{barycentric coordinates} \\ 
A_{1,1} & (1,0,\lambda +1) \\ 
A_{2,2} & (1,\lambda +1,0) \\ 
A_{3,1} & (0,\lambda +1,1)%
\end{array}%
. 
\]

We proceed by computing the equations of the cevians whose points of
intersection are $M$ and $N.$ The cevian $A_{2}A_{2,2}$ has the equation 
\[
\left\vert 
\begin{array}{ccc}
0 & 0 & 1 \\ 
1 & \lambda +1 & 0 \\ 
t_{1} & t_{2} & t_{3}%
\end{array}%
\right\vert =0 
\]%
which is equivalent to $-(\lambda +1)t_{1}+t_{2}=0.$ The cevian $%
A_{1}A_{1,1} $ has the equation 
\[
\left\vert 
\begin{array}{ccc}
0 & 1 & 0 \\ 
1 & 0 & \lambda +1 \\ 
t_{1} & t_{2} & t_{3}%
\end{array}%
\right\vert =0 
\]%
which is equivalent to $-(\lambda +1)t_{1}+t_{3}=0,$ and the cevian $%
A_{3}A_{3,1}$ has the equation 
\[
\left\vert 
\begin{array}{ccc}
1 & 0 & 0 \\ 
0 & \lambda +1 & 1 \\ 
t_{1} & t_{2} & t_{3}%
\end{array}%
\right\vert =0 
\]%
which is equivalent to $-t_{2}+(\lambda +1)t_{3}=0.$ Hence, the point $M$
can be computed by the following system of equations;%
\[
\left\{ 
\begin{array}{c}
-(\lambda +1)t_{1}+t_{2}\ \ \ \ \ \ \ \ \ \ \ \ \ \ \ =0\ \ \ \ \  \\ 
\ \ \ \ \ \ \ \ \ \ \ \ \ -t_{2}+(\lambda +1)t_{3}=0%
\end{array}%
\right. . 
\]

Substituting $t_{2}=1$ we get, 
\[
M=\left( \frac{1}{\lambda +1},1,\frac{1}{\lambda +1}\right) . 
\]

Similarly, the point $N$ is obtained from the following equations;

\[
\left\{ 
\begin{array}{c}
-(\lambda +1)t_{1}\ \ \ \ \ \ \ \ \ \ \ \ \ \ +t_{3}=0 \\ 
-(\lambda +1)t_{1}+t_{2}\ \ \ \ \ \ \ \ \ \ \ \ \ \ =0%
\end{array}%
\right. . 
\]

Substituting $t_{1}=1$ we get, 
\[
N=\left( 1,\lambda +1,\lambda +1\right) . 
\]

Consequently, the ratio of the area of the triangle $GMN$ to the area of the
triangle $A_{2}A_{3}A_{1}$ is equal to the determinant 
\[
\left\vert 
\begin{array}{ccc}
1/3 & 1/3 & 1/3 \\ 
\frac{1}{\lambda +1} & 1 & \frac{1}{\lambda +1} \\ 
1 & \lambda +1 & \lambda +1%
\end{array}%
\right\vert =\frac{\lambda ^{2}}{3(\lambda +1)} 
\]

divided by the product of the sums of the rows;%
\[
\left( 1+\frac{2}{\lambda +1}\right) \left( 2\lambda +3\right) . 
\]

Therefore, the ratio of the area of the hexagon $IGKLMN$ to the area of the
triangle $A_{2}A_{3}A_{1}$ is%
\[
\frac{\frac{6\lambda ^{2}}{3(\lambda +1)}}{\left( 1+\frac{2}{\lambda +1}%
\right) \left( 2\lambda +3\right) }=\frac{2\lambda ^{2}}{\left( \lambda
+3\right) \left( 2\lambda +3\right) }, 
\]

in agreement with the statement of the theorem.
\end{proof}

\subsection{Special cases}

\begin{enumerate}
\item If $\lambda =1$ then the ratio of the area of the hexagon to the area
of the triangle is $\frac{1}{10}.$ This special case is referred to as
Marion Walter's theorem \cite{Walter} which states the following: \textit{If
the trisection points of the sides of any triangle are connected to the
opposite vertices, the resulting hexagon has one-tenth the area of the
original triangle. }

\item If $n$ is odd, $n=2k+1,$ then taking $\lambda =\frac{1}{k}$ implies
that the ratio of the area of the hexagon to the area of the triangle is 
\[
\frac{2}{(3k+1)(3k+2)}=\frac{8}{9n^{2}-1}. 
\]%
This special case is referred to as Morgan's theorem \cite{Morgan}, which
was proved by T. Watanabe, R. Hanson and F. D. Nowosielski \cite{Watanabe}
using Routh-Steiner theorem several times.

\item If $n$ is even, then taking $\lambda =n-2$ implies a new result which
was not mentioned in the previous discussion. In this case, we have the
following:

\textit{The ratio of the area of the hexagon to the area of the triangle is} 
\[
\frac{2(n-2)^{2}}{(n+1)(2n-1)}. 
\]
\end{enumerate}

Note that, if each side of the triangle is divided into $n$ equal parts,
then for odd $n$, $n=2k+1,$ the hexagon is the sub polygon in Theorem \ref%
{general division of triangle}, formed by the cevians $\left\{
A_{i}A_{k,i},A_{i}A_{k+1,i}\right\} ,1\leq i\leq 3.$ While for even $n$, the
hexagon is formed by the cevians $\left\{
A_{i}A_{1,i},A_{i}A_{n-1,i}\right\} ,1\leq i\leq 3.$

\section{Parallelograms}

We can generalize the Routh-Steiner theorem to parallelograms in the
following manner:

\begin{theorem}
If the sides $BC,CD,DA,AB$ of a parallelogram $ABCD$ are divided at $K,L,M,N$
$\ $in the respective ratios $\kappa :1,\lambda :1,$ $\mu :1,$ $\nu :1,$ the
cevians $AK,BL,CM,DN$ form a quadrilateral whose area is 
\begin{eqnarray}
&&\frac{1}{2}\frac{(\frac{\mu }{1+\nu }+\frac{1+\mu }{\lambda })(1+\frac{%
\kappa }{1+\mu }+\frac{\kappa +1}{\nu (1+\mu )})}{(2+\frac{1}{\lambda }-%
\frac{1}{1+\mu })(1+\mu +\frac{\nu \mu }{1+\nu })(1+2\kappa +\frac{\kappa +1%
}{\nu })}+  \nonumber \\
&&\frac{1}{2}\frac{(\frac{\kappa +1}{\nu }+\frac{\kappa }{1+\lambda })(1+%
\frac{1}{\lambda }+\frac{\kappa }{1+\mu })}{(2+\frac{1}{\lambda }-\frac{1}{%
1+\mu })(1+\kappa +\frac{\lambda \kappa }{1+\lambda })(1+2\kappa +\frac{%
\kappa +1}{\nu })}  \label{eq1}
\end{eqnarray}

times that of $ABCD.$
\end{theorem}

\begin{proof}
If the barycentric coordinates of $A,B,C$ are $(0,1,0),(0,0,1),(1,0,0)$
respectively, then the barycentric coordinates of $D$ are $(1,1,-1).$ This
is true since the diagonals in the parallelogram bisect each other and the
barycentric coordinates of the midpoint of $AC$ are $(\frac{1}{2},\frac{1}{2}%
,0)$ (see Fig. \ref{parallel-fig}). The barycentric coordinates of $K,L,M,N$
are shown in the following table:

\[
\begin{array}{cc}
\text{point} & \text{barycentric coordinates} \\ 
K & (\kappa ,0,1) \\ 
L & (\lambda +1,\lambda ,-\lambda ) \\ 
M & (1,\mu +1,-1) \\ 
N & (0,1,\nu )%
\end{array}%
. 
\]

Proceeding as in the proof of Theorem \ref{1-lampda-1-th}, the cevians $%
AK,BL,CM,DN$ have the following equations:%
\[
\begin{array}{cc}
\text{cevian} & \text{barycentric equation} \\ 
AK & -t_{1}+\kappa t_{3}=0 \\ 
BL & -\lambda t_{1}+(\lambda +1)t_{2}=0 \\ 
CM & t_{2}+(\mu +1)t_{3}=0 \\ 
DN & (1+\nu )t_{1}-\nu t_{2}+t_{3}=0%
\end{array}%
. 
\]

Now, let $X,Y,Z,W$ be the points of intersections of pairs of cevians $BL$
and $CM,$ $CM$ and $DN,$ $DN$ and $AK,$ $AK$ and $BL,$ respectively (see
Fig. \ref{parallel-fig}).

Then the barycentric coordinates of these points are given in the following
table:%
\[
\begin{array}{cc}
\text{point} & \text{barycentric coordinates} \\ 
X & (\frac{\lambda +1}{\lambda },1,\frac{-1}{\mu +1}) \\ 
Y & (-\frac{1+\nu +\nu \mu }{1+\nu },-\mu -1,1) \\ 
Z & (\kappa ,\frac{1+\kappa +\kappa \nu }{\nu },1) \\ 
W & (\kappa ,\frac{\lambda \kappa }{\lambda +1},1)%
\end{array}%
. 
\]

The area of the quadrilateral $XYZW$ equals the sum of the areas of the
triangles $XYZ$ and $ZWX.$ Normalizing the barycentric coordinates of $%
X,Y,Z,W$ and dividing by $2$ which is the area of the parallelogram $ABCD,$
we get that the ratio of the area of the quadrilateral $XYZW$ to the area of
the parallelogram $ABCD$ equals $\frac{1}{2}\left( r_{1}+r_{2}\right) $
where, 
\begin{equation}
r_{1}=-\frac{\left\vert 
\begin{array}{ccc}
\frac{\lambda +1}{\lambda } & 1 & \frac{-1}{\mu +1} \\ 
-\frac{1+\nu +\nu \mu }{1+\nu } & -\mu -1 & 1 \\ 
\kappa & \frac{1+\kappa +\kappa \nu }{\nu } & 1%
\end{array}%
\right\vert }{(2+\frac{1}{\lambda }-\frac{1}{1+\mu })(1+\mu +\frac{\nu \mu }{%
1+\nu })(1+2\kappa +\frac{\kappa +1}{\nu })}  \label{eq2}
\end{equation}%
and 
\begin{equation}
r_{2}=\frac{\left\vert 
\begin{array}{ccc}
\kappa & \frac{1+\kappa +\kappa \nu }{\nu } & 1 \\ 
\kappa & \frac{\lambda \kappa }{\lambda +1} & 1 \\ 
\frac{\lambda +1}{\lambda } & 1 & \frac{-1}{\mu +1}%
\end{array}%
\right\vert }{(2+\frac{1}{\lambda }-\frac{1}{1+\mu })(1+\kappa +\frac{%
\lambda \kappa }{1+\lambda })(1+2\kappa +\frac{\kappa +1}{\nu })}.
\label{eq3}
\end{equation}

Note that : 
\[
\left\vert 
\begin{array}{ccc}
\frac{\lambda +1}{\lambda } & 1 & \frac{-1}{\mu +1} \\ 
-\frac{1+\nu +\nu \mu }{1+\nu } & -\mu -1 & 1 \\ 
\kappa & \frac{1+\kappa +\kappa \nu }{\nu } & 1%
\end{array}%
\right\vert =\left\vert 
\begin{array}{ccc}
\frac{\lambda +1}{\lambda } & 1 & \frac{-1}{\mu +1} \\ 
-\frac{1+\nu +\nu \mu }{1+\nu }+(\mu +1)\frac{\lambda +1}{\lambda } & 0 & 0
\\ 
\kappa & \frac{1+\kappa +\kappa \nu }{\nu } & 1%
\end{array}%
\right\vert 
\]

and 
\[
\left\vert 
\begin{array}{ccc}
\kappa & \frac{1+\kappa +\kappa \nu }{\nu } & 1 \\ 
\kappa & \frac{\lambda \kappa }{\lambda +1} & 1 \\ 
\frac{\lambda +1}{\lambda } & 1 & \frac{-1}{\mu +1}%
\end{array}%
\right\vert =\left\vert 
\begin{array}{ccc}
0 & \frac{1+\kappa +\kappa \nu }{\nu } & 1 \\ 
0 & \frac{\lambda \kappa }{\lambda +1} & 1 \\ 
\frac{\lambda +1}{\lambda }+\frac{\kappa }{\mu +1} & 1 & \frac{-1}{\mu +1}%
\end{array}%
\right\vert . 
\]

After evaluating the determinants and substituting in (\ref{eq2}) and (\ref%
{eq3}) the result follows.
\end{proof}

\ 

In particular, if $\kappa =\lambda =\mu =\nu $ then the expressions in the
denominators of (\ref{eq1}) are equal to: 
\[
\frac{\left( 2\lambda ^{2}+2\lambda +1\right) ^{3}}{\lambda ^{2}(\lambda
+1)^{2}}, 
\]

while the expressions in the numerators\ of (\ref{eq1}) are equal to: 
\[
\frac{\left( 2\lambda ^{2}+2\lambda +1\right) ^{2}}{\lambda ^{2}(\lambda
+1)^{2}}. 
\]

Hence, we have the following;

\begin{corollary}
If the sides $BC,CD,DA,AB$ of a parallelogram $ABCD$ are divided at $K,L,M,N$
$\ $in the ratio $\lambda :1,$ then the cevians $AK,BL,CM,DN$ form a
quadrilateral whose area is 
\[
\frac{1}{2\lambda ^{2}+2\lambda +1} 
\]%
times that of $ABCD.$
\end{corollary}

Substituting $\lambda =\frac{1}{p-1},p\geq 2$, we get the equivalent
expression $\frac{p^{2}-2p+1}{p^{2}+1},$ which was discovered by M. De
Villiers \cite{Villiers}. He did not provide any proof but pointed out the
following: \ "Since a square is affinely equivalent to a parallelogram, the
easiest way to derive and prove this formula is to consider the special case
of a square".

We leave to the reader to see what happens in another particular case of the
theorem: if $\kappa =\mu $ and $\lambda =\nu .$

Finally, the fascinating theorem of Routh-Steiner attracted the interest of
mathematics education researchers, probably because of the use of a dynamic
geometry software to rediscover the theorem, but it still attracts the
interest of mathematics researchers as well.

\bigskip

\qquad

\bigskip

\FRAME{ftbphFU}{5.9395in}{3.4169in}{0pt}{\Qcb{The areas of the triangles $%
PA_{2}A_{3},PA_{3}A_{1},PA_{1}A_{2}$ are proportional to the barycentric
coordinates $t_{1},t_{2},t_{3},$ respectively. }}{\Qlb{triangle-fig}}{%
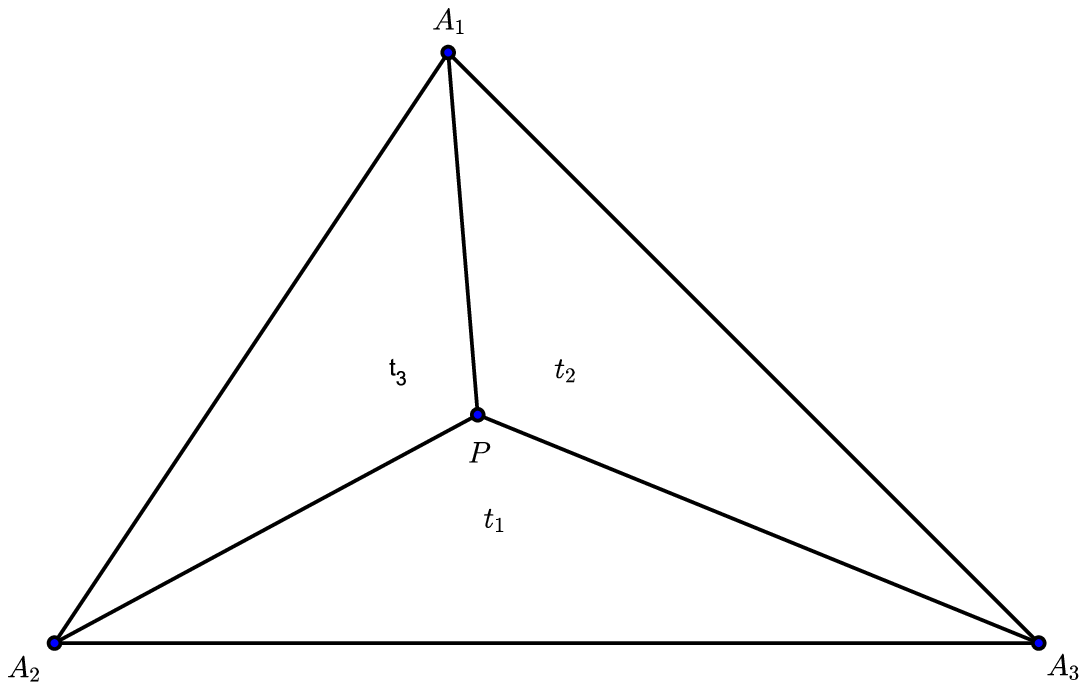}{\special{language "Scientific Word";type
"GRAPHIC";maintain-aspect-ratio TRUE;display "USEDEF";valid_file "F";width
5.9395in;height 3.4169in;depth 0pt;original-width 7.9718in;original-height
4.574in;cropleft "0";croptop "1";cropright "1";cropbottom "0";filename
'../../../../Elias/Marion Walter theorem/triangle.eps';file-properties
"NPEU";}}

\FRAME{ftbpFU}{6.4446in}{3.6806in}{0pt}{\Qcb{The ratio of the area of any
sub-polygon to the area of the whole triangle depends only on $\left\{ 
\protect\lambda _{i,j}\right\} $}}{\Qlb{triangular-network}}{%
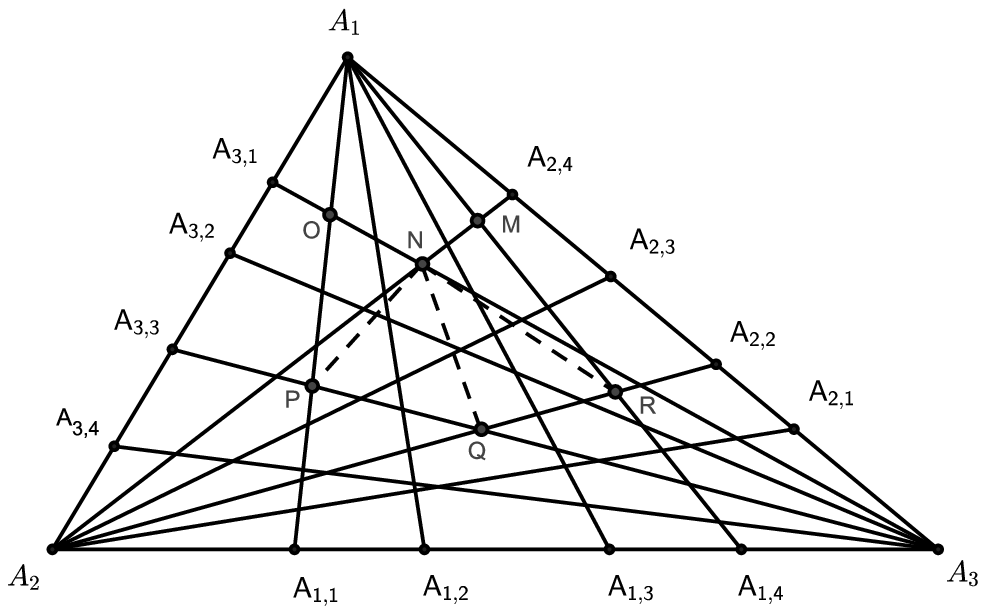}{\special{language "Scientific Word";type
"GRAPHIC";maintain-aspect-ratio TRUE;display "USEDEF";valid_file "F";width
6.4446in;height 3.6806in;depth 0pt;original-width 7.9044in;original-height
4.5031in;cropleft "0";croptop "1";cropright "1";cropbottom "0";filename
'../../../../Elias/Marion Walter
theorem/triangle-general-divisions.eps';file-properties "NPEU";}}

\FRAME{ftbpFU}{7.3613in}{4.2047in}{0pt}{\Qcb{Dividing each side in the
ratios $1:\protect\lambda :1$}}{\Qlb{1-lampda-1}}{%
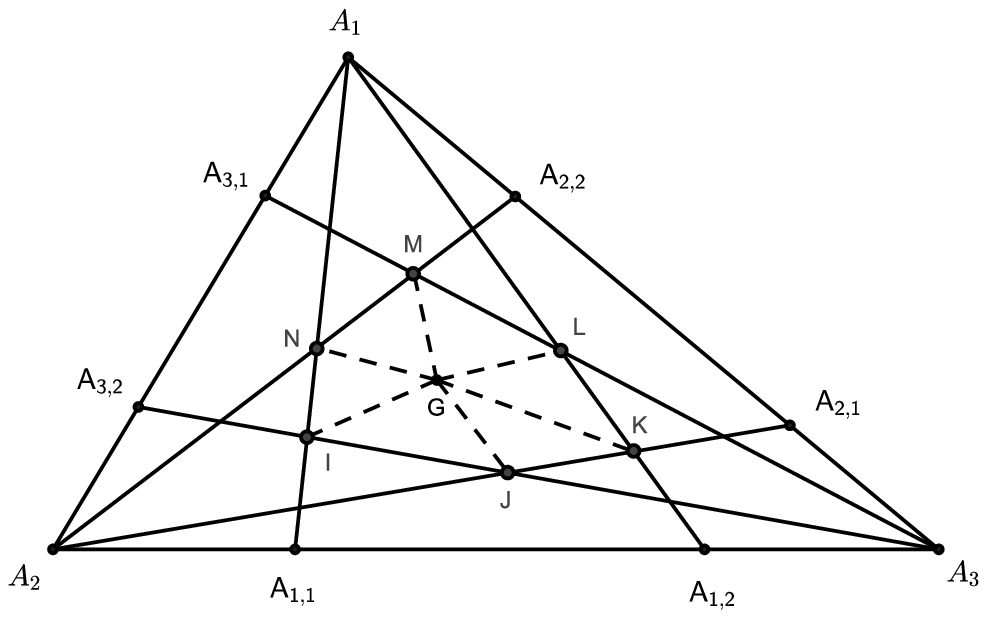}{\special{language "Scientific Word";type
"GRAPHIC";maintain-aspect-ratio TRUE;display "USEDEF";valid_file "F";width
7.3613in;height 4.2047in;depth 0pt;original-width 7.9044in;original-height
4.5031in;cropleft "0";croptop "1";cropright "1";cropbottom "0";filename
'../../../../Elias/Marion Walter
theorem/triangle-division-1-lampda-1.eps';file-properties "NPEU";}}

\FRAME{ftbpFU}{7.0958in}{4.2774in}{0pt}{\Qcb{The sides of a parallelogram
are divided in the ratios $\protect\kappa :1,\protect\lambda :1,$ $\protect%
\mu :1,\protect\nu :1$}}{\Qlb{parallel-fig}}{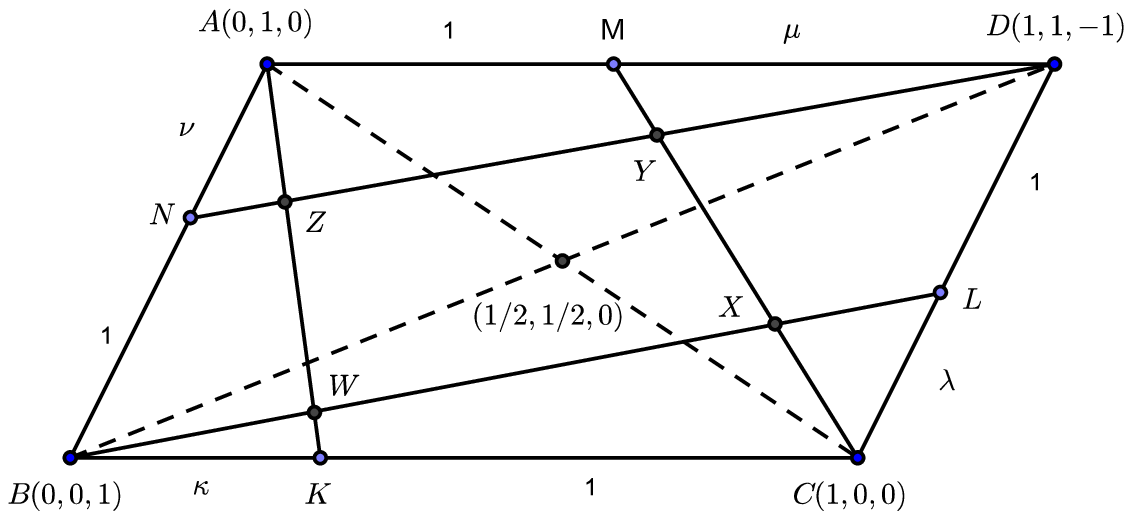}{\special%
{language "Scientific Word";type "GRAPHIC";maintain-aspect-ratio
TRUE;display "USEDEF";valid_file "F";width 7.0958in;height 4.2774in;depth
0pt;original-width 7.9718in;original-height 4.7971in;cropleft "0";croptop
"1";cropright "1";cropbottom "0";filename '../../../../Elias/Marion Walter
theorem/parallelogram.eps';file-properties "NPEU";}}

\end{document}